\documentclass[11pt]{article}

\usepackage{tikz}
\usepackage{subfigure}
\usepackage[english]{babel}
\usepackage{graphicx}

\usepackage[center]{caption2}
\usepackage{amsfonts,amssymb,amsmath,latexsym,amsthm}
\usepackage{multirow}

\newtheorem{thm}{Theorem}[section]

\newtheorem{lem}[thm]{Lemma}

\newtheorem{obs}[thm]{Observation}

\newtheorem{claim}{Claim}

\begin{document}

\title{The degree threshold for covering with all the connected $3$-graphs with $3$ edges
}
\author{Yue Ma$^a$ \quad Xinmin Hou$^{b,c}$ \quad Zhi Yin$^b$\\
\small $^{a}$School of Mathematics and Statistics,\\
\small Nanjing University of Science and Technology,\\
\small Nanjing, Jiangsu 210094, China.\\
\small $^{b}$ School of Mathematical Sciences,\\
\small University of Science and Technology of China,\\
\small Hefei, Anhui 230026, China.\\
\small $^{c}$ CAS Key Laboratory of Wu Wen-Tsun Mathematics,\\
\small University of Science and Technology of China,\\
\small Hefei, Anhui 230026, China.\\
\small $^a$yma@njust.edu.cn; $^b$xmhou@ustc.edu.cn, yinzhi@mail.ustc.edu.cn
}

\date{}

\maketitle

\begin{abstract}
Given two $r$-uniform hypergraphs $F$ and $H$, we say that $H$ has an $F$-covering if every vertex in $H$ is contained in a copy of $F$. 
Let $c_{i}(n,F)$ be the least integer such that every $n$-vertex $r$-graph $H$ with $\delta_{i}(H)>c_i(n,F)$ has an $F$-covering. 
Falgas-Ravry, Markst\"om and Zhao (Combin. Probab. Comput., 2021) asymptotically determined $c_1(n,K_{4}^{(3)-})$, where $K_{4}^{(3)-}$ is obtained by deleting an edge from the complete $3$-graph on $4$ vertices. Later, Tang, Ma and Hou (arXiv, 2022) asymptotically determined $c_1(n,C_{6}^{(3)})$, where $C_{6}^{(3)}$ is the linear triangle, i.e. $C_{6}^{(3)}=([6],\{123,345,561\})$.
In this paper, we determine $c_1(n,F_5)$ asymptotically, where $F_5$ is the generalized triangle, i.e. $F_5=([5],\{123,124,345\})$. 
We also determine the exact values of $c_1(n,F)$, where $F$ is any connected $3$-graphs with $3$ edges and $F\notin\{K_4^{(3)-}, C_{6}^{(3)}, F_5\}$.
\end{abstract}

\section{Introduction}
Given a positive integer $k\ge 2$, a {\it$k$-uniform hypergraph} (or a $k$-graph) $H=(V,E)$ consists of a vertex set $V=V(H)$ and an edge set $E=E(H)\subset \binom{V}{k}$, where $\binom{V}{k}$ denotes the set of all $k$-element subsets of $V$. We write graph for $2$-graph for short. Let $H=(V,E)$ be a simple $k$-graph. For any $S\subseteq V(G)$, let $N_{H}(S)=\{T\subseteq V(H)\backslash S: T\cup S\in E(H)\}$ and the degree $d_{H}(S)=|N_{H}(S)|$. For $1\le i\le k-1$, the {\it minimum $i$-degree} of $H$, denoted by $\delta_i(H)$, is the minimum of $d_H(S)$ over all $S\in\binom{V(H)}{i}$. We also call $\delta_{1}(G)$ the minimum degree of $G$. The $link\; graph$ of a vertex $x$ in $V$, denoted by $H_x$, is a $(k-1)$-graph $H_x=(V(G)\backslash \{x\},N_H(x))$. 

For $r\ge 2$, a {\it complete $r$-graph} on $n$ vertices, denoted by $K_n^{(r)}$, is an $r$-graph on $[n]$ with the edge set $\binom{[n]}{r}$. For a vertex set $V$, we also write $K^{(r)}[V]$ for the complete $r$-graph on $V$. We write $K_n$ for $K_n^{(2)}$ and $K[V]$ for $K^{(2)}[V]$ for short. For an $r$-graph $G$ with $U\subset V(G)$, let $G[U]=(U,E(G)\cap E(K^{(r)}[U]))$ and $G-U=G[V(G)\backslash U]$.
Also, given two $r$-graphs $G$ and $H$, let $G\cup H$ be the vertex-disjoint union of $G$ and $H$. Let $tH:=\bigcup_{i=1}^tG_i$ for some $t\ge 2$ and $r$-graphs $H, G_1,\dots, G_t$ if $G_i\cong H$ for $i\in[t]$.

Given a $k$-graph $F$, we say a $k$-graph $H$ has an $F$-covering if each vertex of $H$ is contained in some copy of $F$. For $1\le i \le k-1$, the {\it $i$-degree threshold} for $F$-covering is defined as
$$c_i(n,F):=\max\{\delta_i(G):G\;\text{is\;a\;$k$-graph\;on\;$n$\;vertices\;with\;no\;$F$-covering}\}.$$ 
We further let the {\it $i$-degree $F$-covering density} be the limit
$$c_i(F):=\lim_{n\to\infty}\frac{c_i(n,F)}{\binom{n-i}{k-i}}.$$

There are two types of extremal problems related to the covering problem. Given a $k$-graph $F$ , a $k$-graph $H$ is $F$-free if $H$ does not contain a copy of $F$ as a subgraph. For For $0\le i \le k-1$, define 
$$ex_{i}(n,F):=\max\{\delta_i(G): G \mbox{ is } F\mbox{-free and } |V(G)|=n\}\mbox{, and } \pi_i(F):=\frac{ex_i(n,F)}{\binom{n-i}{k-i}},$$
where $\delta_{0}(G):=|E(G)|$. The quantities $ex_0(n,F)$ and $\pi_0(F)$ are known as the {\it Tur\'an number} and the {\it Tur\'an density} of $F$ respectively. For Tur\'an problem on hypergraphs, one can refer to a survey given by Keevash~\cite{Ksurvey}.

Given two $k$-graphs $F$ and $H$, an $F$-tiling in $H$ is a spanning subgraph of $H$ consists of vertex-disjoint copies of $F$. For $1\le i\le k-1$ and $n\equiv 0 \mod{|V(F)|}$, define
$$t_i(n,F):=\max\{\delta_i(G):G\;\text{is\;a\;$k$-graph\;on\;$n$\;vertices\;with\;no\;$F$-tiling}\}.$$
The tiling problem in hypergraphs is also widely studied. We recommend a survey given by Zhao~\cite{Zsurvey}.
 
Trivially, for $1\le i\le k-1$,
$$ex_i(n,F)\le c_i(n,F)\le t_i(n,F),$$
which makes the covering problem an interesting but different extremal problem from Tur\'an problem and the tiling problem.

For a graph $F$, the $F$-covering problem was solved asymptotically in \cite{2-graph} by showing that $c_1(F)=\frac{\chi(F)-2}{\chi(F)-1}$, where $\chi(F)$ is the chromatic number of $F$. 

For $r$-uniform hypergraphs with $r\ge 3$, there are also some works related, most of them focus on $r=3$. Here are some exact results for $c_2(n,F)$ and $c_2(F)$ in $3$-graphs.

\begin{itemize}
\item (Falgas-Ravry, Zhao~\cite{co-FZ}) For $n>98$, $c_2(n,K_4^{(3)})=\lfloor\frac{2n-5}{3}\rfloor$.
\item (Yu, Hou, Ma, Liu~\cite{co-Our}) $c_2(n,K_4^{(3)-})=\lfloor\frac{n}{3}\rfloor$ and $c_2(n,K_5^{(3)-})=\lfloor\frac{2n-2}{3}\rfloor$, where $K_{k}^{(r)-}$ ($k\ge r\ge 2$) is an $r$-graph obtained from $K_{k}^{(r)}$ by deleting an edge.
\item (Falgas-Ravry, Zhao~\cite{co-FZ}) $c_2(C_{5}^{(3)})=\frac{1}{2}$, where $C_{5}^{(3)}=([5],\{123,234,345,451,512\})$.
\end{itemize}

For $c_1(n,F)$ and $c_1(F)$ in $3$-graphs, some know results are listed as follows.

\begin{itemize}
\item (Falgas-Ravry, Markstr\"om, Zhao~\cite{FMZ}) $c_1(K_4^{(3)-})=\frac{\sqrt{13}-1}{6}$.
\item (Tang, Ma, Hou~\cite{Our}) $c_1(C_6^{(3)})=\frac{3-2\sqrt{2}}{2}$, where $C_{6}^{(3)}=([6],\{123,345,561\})$.
\item (Falgas-Ravry, Markstr\"om, Zhao~\cite{FMZ}) $\frac{19}{27}\le c_1(K_4^{(3)})\le \frac{19}{27}+7.4\times 10^{-9}$.
\item (Falgas-Ravry, Markstr\"om, Zhao~\cite{FMZ}) $\frac{5}{9}\le c_1(C_{5}^{(3)})\le 2-\sqrt{2}$.
\item (Gu, Wang~\cite{GW}) For $n\ge 5$, $\frac{n^2}{9}\le c_1(n,F_5)\le \frac{n^2}{6}+\frac{5}{6}n-3$, where $F_5=([5],\{123,124,345\})$.
\item (Gu, Wang~\cite{GW}) For $n\ge 8$, $n-2\le c_1(n,LP_3)\le n+4$, where $LP_3=([7],\{123,345,567\})$.
\end{itemize}

In this article, we focus on $3$-graphs with $3$ edges.
Let $H$ be a hypergraph. We say $H$ is {\it connected} if for any pair of vertices $\{u,v\}\subset\binom{V(H)}{2}$, we can find a sequence of edges, say $e_1, e_2, \dots, e_t\in E(H)$, with $u\in e_1$, $v\in e_t$ and $e_i\cap e_{i+1}\neq\emptyset$ for any $i\in[t-1]$. A maximal connected subgraph for any hypergraph $H$ is called a {\it component}. Note that a connected hypergraph consists of a unique component.

By a simple enumeration, one can check that: there are only $9$ kinds of connected $3$-graphs with $3$ edges. We list all of them in Figure~\ref{3edges}.

\begin{figure}[h]
\centering
\subfigure[$K_4^{(3)-}$]{\includegraphics[width=1in]{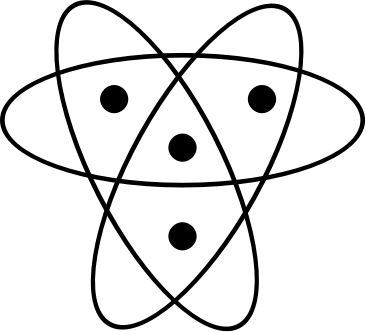}}
\hspace{0.5in}
\subfigure[$C_6^{(3)}$]{\includegraphics[width=1in]{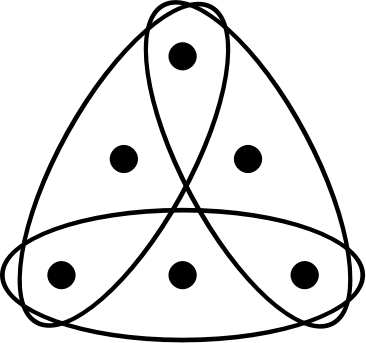}}
\hspace{0.5in}
\subfigure[$F_5$]{\includegraphics[width=1in]{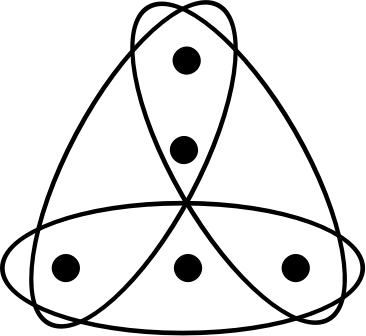}}

\subfigure[$LP_3$]{\includegraphics[width=1in]{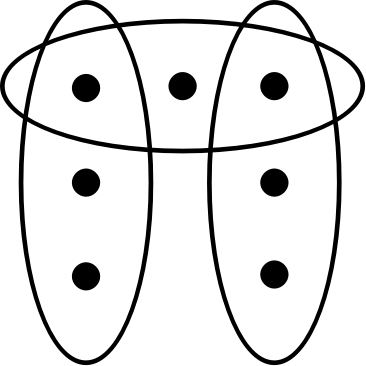}}
\hspace{0.5in}
\subfigure[$TP_3$]{\includegraphics[width=1in]{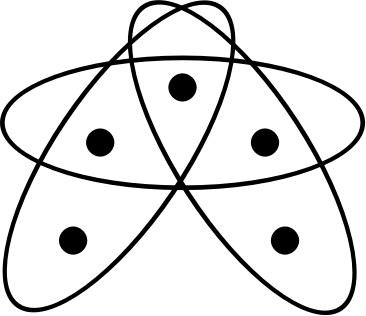}}
\hspace{0.5in}
\subfigure[$GP_3$]{\includegraphics[width=1in]{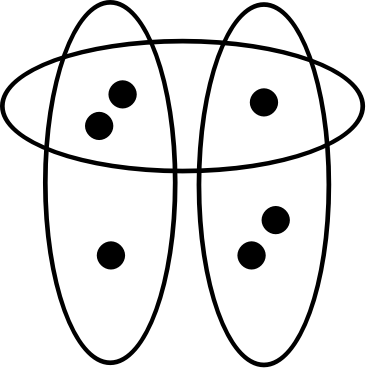}}

\subfigure[$K_{1,1,3}$]{\includegraphics[width=0.9in]{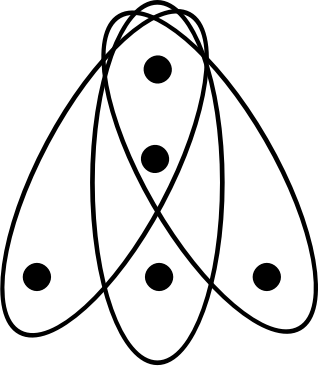}}
\hspace{0.55in}
\subfigure[$S_3$]{\includegraphics[width=1in]{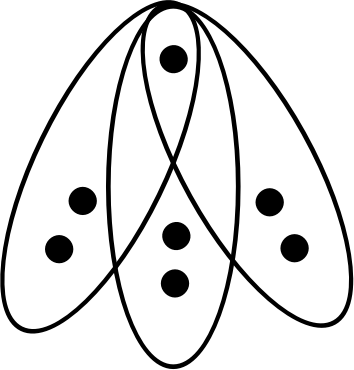}}
\hspace{0.5in}
\subfigure[$GS_3$]{\includegraphics[width=1in]{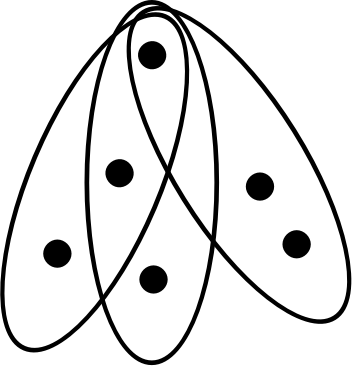}}
\caption{All possible connected $3$-graphs with $3$ edges}\label{3edges}
\end{figure}

In particular, $K_4^{(3)-}$ and $C_{6}^{(3)}$ are two examples for connected $3$-graphs with $3$ edges, whose $1$-degree covering densities are already know as mentioned above.
Another important example is called a {\it generalized triangle}, denonted by $F_5$, which is a $3$-graph on the vertex set $[5]$ with the edge set $\{123,124,345\}$. 
In 1983, Frankl and F\"uredi~\cite{FF83} gave the Tur\'an number for $F_5$.

\begin{thm}[(Frankl, F\"uredi~\cite{FF83})]\label{FF83}
For $n> 3000$, $ex_0(n,F_5)=\lfloor\frac{n}{3}\rfloor\lfloor\frac{n+1}{3}\rfloor\lfloor\frac{n+2}{3}\rfloor$. In particular, $\pi_0(F_5)=\frac{2}{9}$.
\end{thm}
Note that the condition for $n$ in Theorem~\ref{FF83} was later improved to $n>33$ by Keevash and Mubayi~\cite{KM}. There are also some other extremal results related to $F_5$, we refer to \cite{BM,BBHL,BBHLM} for example.

To give the extremal construction for Theorem~\ref{FF83}, we need some definitions.
For two families of sets $\mathcal{A}$ and $\mathcal{B}$, define $\mathcal{A}\vee\mathcal{B}=\{A\cup B: A\in\mathcal{A}\mbox{ and }B\in\mathcal{B}\}$.
For $r\ge 2$, a {\it complete $r$-partite $r$-graph with partition set $V_1,V_2,\dots, V_r$}, denoted by $K[V_1,V_2,\dots,V_r]$, is an $r$-graph on $\bigcup_{i=1}^rV_i$ with the edge set
$$E(K[V_1,V_2,\dots,V_r])=\binom{V_1}{1}\vee\binom{V_2}{1}\vee\dots\vee\binom{V_r}{1}.$$ 
For an $r$-graph $H$ with $\bigcup_{i=1}^rV_i\subset V(H)$, let $G[V_1,\dots, V_r]=(\bigcup_{i=1}^rV_i,E(H)\cap E(K[V_1,\dots,V_r]))$.
If $|V_i|=n_i$ for $i\in [r]$, we write $K_{n_1,\dots, n_r}$ for $K[V_1,\dots,V_r]$. In particular, $K_{1,1}=K_2$ and $K_{1,1,3}=([5],\{123,124,125\})$.

One can check that $K_{\lfloor\frac{n}{3}\rfloor,\lfloor\frac{n+1}{3}\rfloor,\lfloor\frac{n+2}{3}\rfloor}$ on $n$ vertices contains no copy of $F_5$ as its subgraph, which is an extremal construction for Theorem~\ref{FF83}. Hence we can easily deduce from Theorem~\ref{FF83} that
$\pi_1(F_5)=\frac{2}{9}$.
This leads to $c_1(F_5)\ge\frac{2}{9}$. In fact, the result of Gu and Wang~\cite{GW} about $F_5$ implies that $\frac{2}{9}\le c_1(F_5)\le\frac{1}{3}$.
In this paper, we verify the exact value that $c_1(F_5)=\frac{1}{4}$.
\begin{thm}\label{F_5}
For $n\ge 5$, $\frac{1}{8}n^2-\sqrt{2}n<c_1(n,F_5)<\frac{1}{8}n^2+\frac{5}{4}n$. In particular, $c_1(F_5)=\frac{1}{4}$.
\end{thm}

For $k\ge 1$, a {\it linear star} with $k$ edges, denoted by $S_k$, is a $3$-graph on $[2k+1]$ with edge set $\{123,145,167,\dots, 1(2k)(2k+1)\}$. In particular, $S_3=([7],\{123,145,167\})$.

A {\it path of length $k-1$} for some $k\ge 2$, denoted by $P_k$, is a graph on $[k]$ whose edge set is $\{12,23,34,\dots, (k-1)k\}$. In $3$-graph, however, we have several different definitions for a path. 
For $k\ge 1$, a {\it linear $k$-path}, denoted by $LP_k$, is a $3$-graph on $[2k+1]$ with the edge set $\{123, 345, 567, \dots, (2k-1)2k(2k+1)\}$. In particular, $LP_3=([7],\{123,345,567\})$.
For $k\ge 1$, a {\it tight $k$-path}, denoted by $TP_k$, is a $3$-graph on $[k+2]$ with the edge set $\{123, 234, 345, \dots, k(k+1)(k+2)\}$. In particular, $TP_3=([5],\{123,234,345\})$.

There are only two kinds of connected $3$-graphs with $3$ edges other than $K_{4}^{(3)-}$, $C_{6}^{(3)}$, $F_5$, $LP_3$, $TP_3$, $K_{1,1,3}$ and $S_3$. 
We use $GP_3$ and $GS_3$ to denote them:
$$GP_3=([6], \{123, 234, 456\})\mbox{ and } GS_3=([6], \{123,124,156\}).$$

We determine the exact values of $c_1(n,F)$, where $F\in\{LP_3$, $TP_3$, $GP_3$, $K_{1,1,3}$, $S_3$, $GS_3\}$ in this paper.

\begin{thm}\label{other}
(1) For $n\ge 13$, $c_1(n,LP_3)=n-2$.\\
(2) For $n\ge 6$,  
\begin{equation*}
     c_1(n, TP_3)=\begin{cases}
     		n-1  & n\equiv 1  \mod 3;\\
     		n-2  & n\equiv 0,2  \mod 3.
     	\end{cases}
     \end{equation*}
(3) For $n\ge14$, $c_1(n,GP_3)=n-2$.\\
(4) For $n\ge 9$, $c_1(n,K_{1,1,3})=n-1$.\\
(5) For $n\ge 11$, $c_1(n,S_3)=n-1$.\\
(6) For $n\ge 13$, $c_1(n,GS_3)=\lfloor\frac{n-1}{2}\rfloor$.
\end{thm}



The rest of the paper is arranged as follows. In Section 2, we proof Theorem~\ref{F_5}. In Section 3, we show the other cases in turn and finish the proof of Themrem~\ref{other}. We give some concluding remarks in Section 4.


\section{$F_5$: proof of Theorem~\ref{F_5}}
\subsection {Lower bound}
\textbf{Construction 1:} Let $H_1=(V_1,E_1)$ be a $3$-graph with $V_1=\{u\}\sqcup X\sqcup Y\sqcup Z$, and
\begin{eqnarray*}
E_1&=&\left(\{\{u\}\}\vee\binom{X}{1}\vee\binom{Y}{1}\right)\cup\left(\binom{Z}{1}\vee\binom{X}{1}\vee\binom{Y}{1}\right)\\
&&\cup\left(\binom{X}{1}\vee E_X \right)\cup\left(\binom{Y}{1}\vee E_Y \right)\cup\binom{Z}{3}\mbox{,}
\end{eqnarray*}
where $|X|=|Y|=\lfloor\frac{\sqrt{2}}{4}n\rfloor-1$, $E_X\sqcup E_Y=\binom{Z}{2}$ and $||E_X|-|E_Y||\le 1$.

\begin{figure}[h]
\centering
\includegraphics[width=3in]{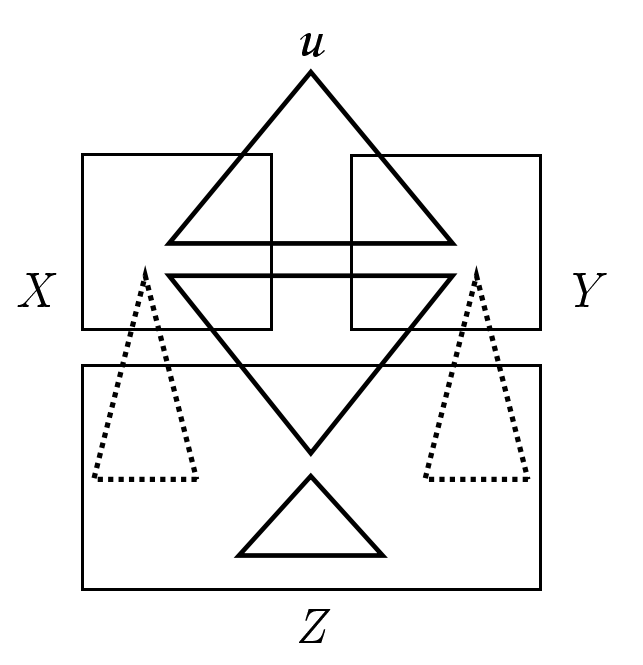}
\caption{Construction 1}\label{c1}
\end{figure}

\begin{obs}
$\delta_1(H_1)>\frac{1}{8}n^2-\sqrt{2}n$ and $H_1$ has no $F_5$ covering $u$.
\end{obs}
\begin{proof}
It is easy to check that $H_1$ has no $F_5$ covering $u$.
Let $a=|X|=|Y|=\lfloor\frac{\sqrt{2}}{4}n\rfloor-1$ and $b=|Z|=n-1-2a$. Since $E_X\sqcup E_Y=\binom{Z}{2}$ and $||E_X|-|E_Y||\le 1$, $|E_X|,|E_Y|\ge\lfloor\frac{1}{2}\binom{b}{2}\rfloor\ge\frac{b(b-1)}{4}-\frac{1}{2}$. Note that the case of $n=5$ is apparently true. For $n\ge 6$, we have $a\ge 1$ and $b\ge 3$. Choose $v\in V(H_1)$.

If $v=u$, then
$$d_{H_1}(v)=a^2>(\frac{\sqrt{2}}{4}n-2)^2=\frac{1}{8}n^2-\sqrt{2}n+4>\frac{1}{8}n^2-\sqrt{2}n\mbox{.}$$

If $v\in X\cup Y$, then
\begin{eqnarray*}
d_{H_1}(v)&\ge& a+ab+\frac{b(b-1)}{4}-\frac{1}{2}=\frac{n^2-3n}{4}-(a+\frac{3}{2})a\\
&\ge& \frac{n^2-3n}{4}-(\frac{\sqrt{2}}{4}n+\frac{1}{2})(\frac{\sqrt{2}}{4}n-1)\\
&=&\frac{1}{8}n^2-\frac{6+\sqrt{2}}{8}n+\frac{1}{2}>\frac{1}{8}n^2-\sqrt{2}n\\
\end{eqnarray*}

If $v\in Z$, then
$$d_{H_1}(v)>a^2+\binom{b-1}{2}>a^2=d_{H}(u)>\frac{1}{8}n^2-\sqrt{2}n\mbox{.}$$
Therefore, $\delta_1(H_1)>\frac{1}{8}n^2-\sqrt{2}n$.
\end{proof}

\subsection{Upper bound}
For any graph $G$, let $\mathcal{E}(G)=\{uv\in\binom{V(G)}{2}: N_G(v)\cap N_{G}(u)\neq\emptyset\}$ be the graph on $V(G)$ whose edges are all pairs of vertices sharing as least one common neighbor. We have the following result about the number of edges in $\mathcal{E}(G)$.

\begin{lem}\label{2}
For any graph $G$ on $n$ vertices, $|E(\mathcal{E}(G))|\ge |E(G)|-\frac{n}{2}$.
\end{lem}
\begin{proof}
We prove by induction on $n$.
Firstly, for $1\le n\le 3$, the inequation is apparently true. Now let $G$ be a graph on $n\ge 4$ vertices and suppose the inequation holds for any graph on less than $n$ vertices.
If $G$ is an empty graph, we are done. Otherwise, pick an edge $uv\in E(G)$. By deleting the vertices $u,v$ and all the incidence edges, we get
$$|E(G-\{u,v\})|=|E(G)|-d_{G}(u)-d_{G}(v)+1\mbox{.}$$
On the other hand, the deletion must destroy all the edges incident with one of $u$ and $v$ in $\mathcal{E}(G)$. Note that $u$ (resp. $v$) incident with all the vertices in $N_{G}(v)-\{u\}$ (resp. $N_{G}(u)=\{v\}$) within $\mathcal{E}(G)$. In other words,
$$|E(\mathcal{E}(G-\{u,v\}))|\le |E(\mathcal{E}(G))|-d_{G}(u)-d_{G}(v)+2\mbox{.}$$
Therefore, by induction,
\begin{eqnarray*}
|E(\mathcal{E}(G))|&\ge&|E(\mathcal{E}(G-\{u,v\}))|+d_{G}(u)+d_{G}(v)-2\\
&\ge& |E(G-\{u,v\})|-\frac{n-2}{2}+d_{G}(u)+d_{G}(v)-2\\
&=&|E(G-\{u,v\})|+d_{G}(u)+d_{G}(v)-1-\frac{n}{2}\\
&=&|E(G)|-\frac{n}{2}\mbox{.}
\end{eqnarray*}
This completes the proof.
\end{proof}

\begin{proof}[Proof of Theorem~\ref{F_5}]
It is sufficient to show that every $3$-graph $H$ on $n$ vertices with $\delta_1(H)\ge \frac{1}{8}n^2+\frac{5}{4}n$ has an $F_5$-covering.

Suppose the contrary that there is a $3$-graph $H$ on $n$ vertices with $\delta_1(H)\ge \frac{1}{8}n^2+\frac{5}{4}n$ and a vertex $u\in V(H)$ is not contained in any copy of $F_5$ in $H$. By definition, the link graph $H_u$ contains at least $\delta_1(H)$ edges, so it is not empty. We have the following key claim.
\begin{claim}\label{key}
Let $xy\in E(H_u)$ be an edge in $H_u$, then the four sets $E(H_u-\{x,y\})$, $E(H_x-\{u\})$, $E(H_y-\{u\})$ and $E(\mathcal{E}(H_u-\{x,y\}))$ are pairwise disjoint.
\end{claim}
\begin{proof}
(i) If $E(H_u-\{x,y\})\cap E(H_x-\{u\})\neq\emptyset$, we pick a pair $ab$ in it. By definition, $abx, abu, uxy\in E(H)$, which form a copy of $F_5$, a contradicition. The same thing holds for $E(H_u-\{x,y\})$ and $E(H_y-\{u\})$.\\
(ii) If $E(H_x-\{u\})\cap E(H_y-\{u\})\neq \emptyset$, we pick a pair $ab$ in it. Then $abx, aby, xyu$ form a copy of $F_5$, which is a contradiction.\\
(iii) To show $E(H_u-\{x,y\})\cap E(\mathcal{E}(H_u-\{x,y\}))=\emptyset$, we only need to show that there is no triangle in $H_u$. By (i) and (ii), $E(H_u-\{x,y\})$, $E(H_x-\{u\})$ and $E(H_y-\{u\})$ are pairwise disjoint for any $xy\in E(H_u)$.
If there is a traingle $\{xy,xz,yz\}\subset E(H_u)$, then it is easy to see that $E(H_x-\{u\})$, $E(H_y-\{u\})$, $E(H_z-\{u\})$ and $E(H_u-\{x,y,z\})$ are pairwise disjoint since $E(H_u-\{x,y,z\})=E(H_u-\{x,y\})\cap E(H_u-\{x,z\})\cap E(H_u-\{y,z\})$.
This means 
$$|E(H_u-\{x,y,z\})|+|E(H_x-\{u\})|+|E(H_y-\{u\})|+|E(H_z-\{u\})|\le |\binom{V(H_u)}{2}|=\binom{n-1}{2}\mbox{.}$$
Also, $|E(H_u-\{x,y,z\})|\ge \delta_1(H)-(3n-6)$ and $|E(H_w-\{u\})|\ge\delta_1(H)-(n-1)$ for $w\in\{x,y,z\}$. This gives $4\delta_1(H)-((3n-6)+3(n-1))\le\binom{n-1}{2}$, a contradiction by $\delta_1(H)\ge \frac{1}{8}n^2+\frac{5}{4}n$.\\
(iv) If $E(H_x-\{u\})\cap E(\mathcal{E}(H_u-\{x,y\}))\neq\emptyset$, we pick a pair $ab$ in it. By the definition of $\mathcal{E}(H_u-\{x,y\})$, there exists a vertex $c$ with $ac,bc\in E(H_u-\{x,y\})$. Thus, $uca, ucb, abx$ form a copy of $F_5$, a contradicition. The same thing holds for $E(H_y-\{u\})=\emptyset$ and $E(\mathcal{E}(H_u-\{x,y\}))$. 
\end{proof}
Pick an edge $xy\in E(H_u)$. It is easy to check that $|E(H_u-\{x,y\})|\ge\delta_1(H)-(2n-3)$ and $|E(H_x-\{u\})|\ge \delta_1(H)-(n-1)$. By Lemma~\ref{2}, 
$$|E(\mathcal{E}(H_u-\{x,y\}))|\ge |E(H_u-\{x,y\})|-\frac{n-3}{2}\ge\delta_1(H)-(\frac{5}{2}n-\frac{9}{2})\mbox{.}$$
By Claim~\ref{key}, $E(H_u-\{x,y\})$, $E(H_x-\{u\})$, $E(H_y-\{u\})$ and $E(\mathcal{E}(H_u-\{x,y\}))$ are pairwise disjoint. This means
$$|E(H_u-\{x,y\})|+|E(H_x-\{u\})|+|E(H_y-\{u\})|+|E(\mathcal{E}(H_u-\{x,y\}))|\le |\binom{V(H_u)}{2}|=\binom{n-1}{2}\mbox{.}$$
Thus,
$$4\delta_1(H)-(2n-3)-2(n-1)-(\frac{5}{2}n-\frac{9}{2})\le\binom{n-1}{2}\mbox{,}$$
a contradiction by $\delta_1(H)\ge \frac{1}{8}n^2+\frac{5}{4}n$.
\end{proof}

\section{Other cases: proof of Theorem~\ref{other}}
\subsection{$LP_3$}
\begin{proof}[Proof of (1)]
For the lower bound, we simply consider the following $3$-graph $G$ called a {\it trivial intersecting family} on $V(G)=\{0\}\cup[n-1]$ with edge set $E(G)=\{\{0\}\}\vee\binom{[n-1]}{2}$.

For the upper bound, suppose the contrary that there is a $3$-graph $H$ on $n\ge 13$ vertices with $\delta_1(H)\ge n-1$ while some vertex $u\in V(H)$ is not contained in any copy of $LP_3$ in $H$.
\begin{claim}\label{LP_3-1}
We can find a copy of $K_{1,2}\cup K_{1,1}$ in the graph $H_u$.
\end{claim}
\begin{proof}
Note that $H_u$ is a graph on $n-1$ vertices with at least $\delta(H)\ge n-1$ edges. Pick $v\in V(H_u)$ with the maximum degree $d$ of $H_u$. 
By Handshaking Lemma, $(n-1)d\ge\sum_{x\in V(H_u)}d_{H_u}(x)=2E(H_u)\ge 2(n-1)$. Thus, $d\ge 2$, and if $d=2$, then $d_{H_u}(x)=2$ for all the vertex $x\in V(H_u)$.
Now suppose $d=2$ and $N_{H_u}(v)=\{x,y\}$. Since $d_{H_u}(x)=d_{H_u}(y)=2$, there are at most $2+2=4$ edges incident with at least one of $x$ and $y$ in $H_u$. Hence we can pick an edge $ab\in E(H_u-\{x,y\})$ since $|E(H_u-\{x,y\})|\ge |E(H_u)|-4>0$. Clearly, $v\neq a,b$. This means the three edges $vx,vy,ab$ form a $K_{1,2}\cup K_{1,1}$ in $H_u$, a contradiction.
Suppose $d\ge 4$. Pick $4$ vertices $w,x,y,z\in N_{H_u}(v)$. Since $d\le |V(H_u)|-1=n-2$, there exists at least $|E(H_u)|-(n-2)\ge 1$ edge $ab\in E(H_{u}-\{v\})$. Note that $v\neq a,b$ and at least two of $w,x,y,z$ are not contained in $\{a,b\}$. Without loss of generality, suppose $w,x\notin \{a,b\}$, then the three edges $vw,vx,ab$ together form a $K_{1,2}\cup K_{1,1}$ in $H_u$, a contradiction.
Thus, $d=3$. Let $N_{H_u}(v)=\{x,y,z\}$. It is easy to see that there are at most $\binom{3}{2}=3$ edges contained in $\{x,y,z\}$. Note that $|E(H_u-\{v\})|=|E(H_u)|-3>3$. We can pick an edge $ab\in E(H_u-\{v\})$ which is not contained in $\{x,y,z\}$. In other words, $|\{a,b\}\cap \{x,y,z\}|\le 1$, so we can pick two vertices in $\{x,y,z\}$, say $x$ and $y$, which are not in $\{a,b\}$. Hence, the three edges $vx,vy,ab$ form a $K_{1,2}\cup K_{1,1}$ in $H_u$, a contradiction.
\end{proof}
\begin{claim}\label{LP_3-2}
We can find a copy of $K_{1,3}\cup K_{1,2}$ in the graph $H_u$.
\end{claim}
\begin{proof}
By Claim~\ref{LP_3-1}, we choose a set of $5$ vertices $\{a, a_1, a_2, b, b_1\}\subset V(H_u)$ with $aa_1, aa_2, bb_1 \in E(H_u)$. 
We claim that $H_{a_1}-\{u\}\subset\binom{\{a,a_2,b,b_1\}}{2}$. 
Otherwise, there exists an edge $xy\in E(H_{a_1}-\{u\})$ with $|\{x,y\}\cap \{a,a_2,b,b_1\}|\le 1$.
If $\{x,y\}\cap\{a,a_2,b,b_1\}=\emptyset$, then $\{xya_1$, $a_1au$, $ubb_1\}$ is a copy of $LP_3$ in $H$, a contradcition. Thus, exactly one of $a,a_2,b,b_1$ is contained in $\{x,y\}$. Without loss of generality, suppose that $x$ is this vertex.
If $x=a$, then $\{a_1ya$, $aa_2u$, $ubb_1\}$ is a copy of  $LP_3$; If $x=a_2$, then $\{a_1ya_2,a_2au,ubb_1\}$ is a copy of  $LP_3$; If $x=b$, then $\{a_1yb,bb_1u,uaa_2\}$ is a copy of  $LP_3$; If $x=b_1$, then $\{a_1yb_1,b_1bu,uaa_2\}$ is a copy of  $LP_3$. Any of the four cases leads to a contradiction.
Therefore, $H_{a_1}-\{u\}\subset\binom{\{a,a_2,b,b_1\}}{2}$. In particular, $|E(H_{a_1}-\{u\})|\le\binom{4}{2}=6$ and then $d_{H_u}(a_1)=d_{H_{a_1}}(u)=|E(H_{a_1})|-|E(H_{a_1}-\{u\})|\ge (n-1)-6\ge 6$. Similarly, $d_{H_u}(a_2)\ge 6$.
Now, pick three vertices $c_1,c_2,c_3\in N_{H_u}(a_1)\backslash\{a_2\}$, then we still have at least $6-1-3=2$ vertices $c_4,c_5\in N_{H_u}(a_2)\backslash\{a_1,c_1,c_2,c_3\}$. This gives $5$ edges $a_1c_1, a_1c_2, a_1c_3, a_2c_4,a_2c_5\in E(H)$, which form a $K_{1,3}\cup K_{1,2}$ in $H_u$. 
\end{proof}
Now by Claim~\ref{LP_3-2}, we can choose a set of $7$ vertices $\{a, a_1, a_2, a_3, b, b_1, b_2\}\subset V(H_u)$ with $aa_1, aa_2, aa_3, bb_1, bb_2\in E(H_u)$. Similarly as the proof in Claim~\ref{LP_3-2}, one can check by simple discussions that, for $i=1,2,3$, $E(H_{a_i}-\{u\})\subseteq\{ab\}$. This means $|E(H_{a_i}-\{u\})|\le 1$ and then $d_{H_u}(a_i)=d_{H_{a_i}}(u)=|E(H_{a_i})|-|E(H_{a_i}-\{u\})|\ge (n-1)-1\ge n-2$ for any $i\in[3]$. Hence for $i\in[3]$, $ab\in E(H_{a_i}-\{u\})$ and $a_iv\in E(H_u)$ for any $v\in V(H_u)\backslash\{a_i\}$. In particular, $a_1ab, ua_2a_3\in E(H)$. Together with the edge $ub_1b\in E(H)$, we get a copy of $LP_3$ in $H$ covering $u$, a contradiction.
\end{proof}

\subsection{$TP_3$}
\begin{proof}[Proof of the lower bound of (2)]
For $n\equiv 0,2 \mod 3$, consider the $3$-graph $F_{n-2,2}$ on $[n]$ with the edge set $\{\{n-1,n\}\}\vee\binom{[n-2]}{1}\cup\binom{[n-2]}{3}$. one can check that $\delta_1(F_{n-2,2})=n-2$ for $n\ge 6$ and there is no copy of $TP_3$ containing the vertex $n$ in $F_{n-2,2}$.\\
For $n\equiv 1 \mod 3$, suppose $n=3k+1$ for some integer $k\ge 2$. Consider a $3$-graph $F$ on the vertex set $\{u\}\cup\bigcup_{i=1}^kA_i$ with $|A_i|=3$ for any $i\in[k]$. The edge set of $F$ is 
$$E(F)=\bigcup_{i=1}^k\left(\{u\}\vee\binom{A_i}{2}\right)\cup\bigcup_{\{i,j,k\}\in\binom{[k]}{3}}\left(\binom{A_i}{1}\vee\binom{A_j}{1}\vee\binom{A_k}{1}\right)\mbox{.}$$
One can also check that $\delta_1(F)=d_{F}(u)=3k=n-1$ and $F$ has no copy of $TP_3$ containing $u$.
\end{proof}
\begin{proof}[Proof of the upper bound of (2)]
Let $g(n)$ be a function with $g(n)=n-1$ for $n\equiv 0,2 \mod 3$ and $g(n)=n$ for $n\equiv 1 \mod 3$.
Suppose the contrary that there is a $3$-graph $H$ on $n\ge 6$ vertices with $\delta_1(H)\ge g(n)$ and there is a vertex $u\in V(H)$ which is not contained in any copy of $TP_3$ in $H$.

Apparently, there is no copy of $P_4$ contianed in $H_u$. Otherwise, there must be $4$ vertices $x_1, x_2, x_3, x_4 \in V(H_u)$ with $x_1x_2, x_2x_3, x_3x_4\in E(H_u)$, and we can pick $\{x_1x_2u, x_2ux_3, ux_3x_4\}$ as a copy of $TP_3$ in $H$, a contradiction. This implies that any component of $H_u$ can only be a $K_{1,t}$ for some $t\ge 0$ or a $K_3$. Let $n_t$ be the number of components isomorphic to $K_{1,t}$ for any $t\ge 0$ and let $m$ be the number of components isomorphic to $K_3$ in $H$. Then $n-1=3m+\sum_{t\ge 0}(t+1)n_t$ and
$$d_{H}(u)=|E(H_u)|=3m+\sum_{t\ge 0}tn_t=n-1-\sum_{t\ge 0}n_t\mbox{.}$$
If there exists some $i\ge 0$ with $n_i\neq 0$, then $d_{H}(u)\le n-2<g(n)$, a contradiction.
Thus, $n_i=0$ for any $i\ge 0$ and $n=3m+1\equiv 1 \mod 3$. This means $d_{H}(u)=3m=n-1<n=g(n)$, a contradiction, too.
\end{proof}

\subsection{$GP_3$}
\begin{proof}[Proof of the lower bound of (3)]
We consider the same $3$-graph as mentioned in the proof of (1), i.e, consider a trivial intersecting family $G$ on $V(G)=\{0\}\cup[n-1]$ with edge set $E(G)=\{\{0\}\}\vee\binom{[n-1]}{2}$. Apparently, $\delta_1(G)=n-2$ and $G$ contains no copy of $GP_3$ covering $0$.
\end{proof}

\begin{proof}[Proof of the upper bound of (3)]
Let $H$ be a $3$-graph on $n\ge 14$ vertices and $\delta_1(H) \ge n-1$, Let $M\subset V(H)$ be the set of all vertices not covered by any copy of $GP_3$ in $H$. Take $u\in M$ with $d_{H}(u)\le d_{H}(v)$ for all $v\in M$.

In graph theory, a {\it cycle} of length $t$ is a graph on $[t]$ with edge set $\{12,23,\dots, (t-1)t,t1\}$. 
\begin{claim}\label{GP_3-1}
$H_u$ does not contain $K_{1,3} \cup K_2$ as a subgraph. Moreover, $H_u$ is a $2$-regular graph ($d_{H_u}(x)=2$ for any $x\in V(H_u)$), i.e. $H_u$ is the union of some vertex-disjoint cycles on $n-1$ vertices. 
\end{claim}
\begin{proof}
Suppose the contrary that there exist $a,a_1,a_2,a_3,b_1,b_2\in V(H_u)$ with $aa_1$, $aa_2$, $aa_3$, $b_1b_2\in E(H_u)$. 
Let $e\in E(H-\{u\})$ be an edge with $a_1\in e$.

If $e\neq a_1a_2a_3$, then one of $a_2$ and $a_3$, say $a_2$, has $a_2\notin e$. Then if $a\notin e$, $a_2ua$, $uaa_1$, $e$ form a copy of $GP_3$, a contradiction. Now suppose $e\neq a_1a_2a_3$, then $a\in e$. 
Then if $e\cap\{b_1,b_2\}=\emptyset$, $e$, $aa_1u$, $ub_1b_2$ form a copy of $GP_3$. 
Hence we can conclude that $e\in\{a_1a_2a_3, a_1ab_1, a_1ab_2\}$ and $d_{H-\{u\}}(a_1)\le 3$. Similarly, $d_{H-\{u\}}(a_2),d_{H-\{u\}}(a_3)\le 3$.

Recall that $d_{H}(u)\le d_{H}(v)$ for all $v\in M$. If $d_{H-\{u\}}(a)=0$, then all edges containing $a$ must also contain $u$, which means $a\in M$. However, this also implies that $d_{H}(a)=d_{H_u}(a)<|E(H_u)|=d_{H}(u)$ since $a\notin b_1b_2\in E(H_u)$, which leads to a contradiction by the minimality of $d_{H}(u)$.
Hence, $d_{H-\{u\}}(a)\ge 1$. So we can pick $f\in E(H-\{u\})$ with $a\in f$. Thus one of $a_1$, $a_2$, $a_3$, say $a_1$, has $a_1\notin f$. 
If $d_{H_u}(a_1)\ge 4$, then we can pick $c\in N_{H_u}(a_1)\backslash f$. Then $uca_1$, $ua_1a$, $f$ form a copy of $GP_3$. Thus, $d_{H_u}(a_1)\le 3$.
Therefore, $d_{H}(a_1)=d_{H-\{u\}}(a_1)+d_{H_u}(a_1)\le 3+3=6<\delta_1(H)$, a contradiction. 

Now $H_u$ is a $K_{1,3}\cup K_2$-free graph on $n-1$ vertices with at least $n-1$ edges. If $H_u$ does not contain a vertex of degree at least $3$, then it is easy to see that $H_u$ must be $2$-regular and we are done. 
Otherwise, pick $v\in V(H_u)$ with at least $3$ vertices $v_1, v_2, v_3\in N_{H_u}(v)$. 
Apparently, all the the edges in $H_u$ must incident with $V_0:=\{v,v_1,v_2,v_3\}$ or we get a copy of $K_{1,3}\cup K_2$. In other words, $N_{H_u}(x)\subset V_0$ for any $x\in V(H_u)\backslash V_0$.
Also note that $|E(H_u)|\ge n-1>\binom{5}{2}$, we have at least $6-4=2$ vertices, say $x_1$ and $x_2$,  other than $v$, $v_1$, $v_2$ and $v_3$ incident with at least one edge in $H_u$.
If $x_1v\in E(H_u)$ and some vertex in $V_0\backslash\{v_0\}$, say $v_1$ has $x_2v_1\in E(H_u)$, then $x_2v_1$, $vv_2$, $vv_3$, $vx_1\in E(H_u)$ form a copy of $K_{1,3}\cup K_2$, a contradiction. 
Thus, if $x_1v\in E(H_u)$, then $x_2v\in E(H_u)$, which then implies that $N_{H_u}(x)\subseteq\{v\}$ for any $x\in V(H_u)\backslash\{v\}$. This gives $|E(H_u)|\le n-2<n-1$, a contradiction. Hence, $N_{H_u}(x)\subset V_1=\{v_1,v_2,v_3\}$ for any $x\in V(H_u)\backslash V_0$.
If there exists some $i\in[3]$ with $v_ix_1,v_ix_2\in E(H_u)$, then $v_iv$, $v_ix_1$, $v_ix_2$ and $vv_j$ for some $j\neq i$ form a copy of $K_{1,3}\cup K_2$ in $H_u$, a contradiction. Thus, $|N_{H_u}(v_i)\backslash V_0|\le 1$ for $i\in [3]$, which gives $|E(H_u)|\le\binom{4}{2}+3=9<n-1$, a contradiction.
\end{proof}

\begin{claim}\label{GP_3-2}
For any cycle $C\subset H_u$ and edge $e\in E(H_u)$, we have $|V(C) \cap e|\in\{0,3\}$.
\end{claim}
\begin{proof}
Suppose $|V(C) \cap e|=1$ firstly.
Let $V(C)=\{c_1,c_2,...,c_\ell\}$, $E(C)=\{c_1c_2$, $c_2c_3$, $\dots$, $c_{\ell-1}c_l$, $c_\ell c_1\}$ and let $e=c_1xy$ where $x,y \notin C$. Then $c_1c_2c_3$, $c_{\ell-1}c_\ell c_1$, $c_1xy$ form a copy of $GP_3$ covering $u$, a contradiction. So $|V(C) \cap e|\neq 1$ for any cycle $C\subset H_u$.
If $|V(C) \cap e|=2$, then there must exist another cycle $C'$ with $|V(C')\cap e|=3-2=1$, a contradiction. 
\end{proof}
Pick a cycle $C_0$ with $V(C_0)=\{c_1,c_2,\dots,c_\ell\}$ and $E(C_0)=\{c_1c_2,c_2c_3,\dots ,c_{\ell-1}c_\ell,c_\ell c_1\}$.

If $\ell=|V(C_0)|\ge 7$, we pick an edge $e$ with $e\cap V(C_0)\neq\emptyset$ (such edge exists since the degree of vertex in $V(C_0)$ should be more than $2$ in $H$). Then $|e \cap V(C_0)|=3$ by Claim~\ref{GP_3-2}. Suppose $e=\{c_i, c_j, c_k\}$ with $1\le i<j<k\le\ell$. By Pigeonhole Principle, one of $d_1=j-i$, $d_2=k-j$, $d_3=\ell+i-k$, say $d_1$, has $d_1\ge\lceil\ell/3\rceil\ge3$.
This means $j-i\ge 3$.
Without loss of generality, suppose $i=1$, so $k>j\ge 4$.
Then  $c_3uc_2$, $uc_2c_1$, $e$ form a copy of $GP_3$ covering $u$.

Therefore, $|V(C_0)|\le 6$.
Pick $v\in V(C_0)$. Note that any edge $e$ containing $v$ must have $|e\cap V(C_0)|=3$, which implies that $d_{H}(v)\le 2+\binom{|V(C_0)|-1}{2}\le 12<n-1\le \delta_1(H)$. This is a contradiction.
\end{proof}

\subsection{$K_{1,1,3}$}
\begin{proof}[Proof of the lower bound of (4)]
Let $W$ be a $3$-graph on $[n]$ and let $\mathcal{C}=\{12, 23, \dots,$ $(n-2)(n-1), (n-1)1\}$. The edge set of $W$ is 
$$E(W)=\left(\{\{n\}\}\vee\mathcal{C}\right)\cup\left\{\{i,j,k\}\in\binom{[n-1]}{3}: \binom{\{i,j,k\}}{2}\cap\mathcal{C}=\emptyset\right\}.$$
It is easy to see that $d_{W}(n)=n-1$ and $d_{W}(i)=\binom{n-4}{2}-(n-5)+2\ge n-1$ for $i\in[n-1]$ since $n\ge 9$. Hence $\delta_1(W)=n-1$. Also, one can check that there is no copy of $K_{1,1,3}$ covering the vertex $n$. 
\end{proof}
\begin{proof}[Proof of the upper bound of (4)]
Suppose the contrary that there is a $3$-graph $H$ on $n\ge 9$ vertices with $\delta_1(H)\ge n$ and $u\in V(H)$ is not contained in any copy of $K_{1,1,3}$ in $H$. 
Then the degree of any vertex in $H_u$ must be at most $2$. Otherwise, suppose $d_{H_u}(v)\ge 3$ for some $v\in V(H_u)$. Pick $x,y,z\in N_{H_u}(v)$, we get the three edges $uvx,uvy,uvz$ in $H$ which form a $K_{1,1,3}$ in $H$, a contradiction. Thus, $d_{H_u}(v)\le 2$ for any $v\in V(H_u)$.
Note that $V(H_u)=n-1$ and $|E(H_u)|\ge\delta_1(H)\ge n$. By Handshaking Lemma, $2(n-1)\ge\sum_{v\in V(H_u)}d_{H_u}(v)=2|E(H_u)|\ge 2n$, a contradiction
\end{proof}

\subsection{$S_3$}
\begin{proof}[Proof of the lower bound of (5)]
Let $S$ be a $3$-graph on $[n]$ with the edge set
$$E(S)=\left(\{\{n-1\}\}\vee\binom{\{n-2,n-3\}}{1}\vee\binom{[n-4]}{1}\right)\cup\left(\{\{n\}\}\vee\binom{[n-2]}{2}\right).$$
Note that $n\ge 11$. It is easy to check that $d_{S}(n)=\binom{n-2}{2}>n-1$, $d_{S}(n-1)=2(n-4)>n-1$, $d_{S}(n-2)=d_{S}(n-3)=2n-7>n-1$ and $d_{S}(i)=n-1$ for $i\in[n-4]$. This means $\delta_1(S)=n-1$. Also, $S$ has no copy of $S_3$ covering the vertex $n-1$.
\end{proof}

Before the proof of the upper bound, we firstly put the famous Tutte-Berge Theorem here.
\begin{lem}[\cite{Tutte-Berge1}, see also \cite{Tutte-Berge2}]\label{Tutte-Berge}
A graph $G$ is $(s+1)K_2$-free if and only if there is a set $B\subset V(G)$, such that the vertex sets of all the connected components $G_1,\cdots,G_m$ of $G-B$ have $|V(G_i)|\equiv 1 \mod 2$ ($i\in[m]$), and we have,
$$|B|+\sum_{i=1}^{m}\frac{|V(G_i)|-1}{2}=s\quad\mbox{  and  }\quad |B|+\sum_{i=1}^m|V(G_i)|=n\mbox{.}$$
\end{lem}
\begin{proof}[Proof of the upper bound of (5)]
Suppose the contrary that $H$ is a $3$-graph on $n\ge 11$ vertices with $\delta_1(H)\ge n$ and $u\in V(H)$ is not contained in any copy of $S_3$ in $H$. 
Note that there is no copy of $3K_2$ in $H_u$. Ohterwise, let $\{a_1a_2, b_1b_2, c_1c_2\}\subset H_u$ be a copy of $3K_2$, then$\{ua_1a_2, ub_1b_2, uc_1c_2\}$ is a copy of $S_3$ in $H$, a contradiction.
Hence, we can use Lemma~\ref{Tutte-Berge} to obtain a set $B\subset V(H_u)$. Then all the components $G_1,\dots G_m$ of $G-B$ have $|V(G_i)|\equiv 1 \mod 2$ ($i\in[m]$), and
$$|B|+\sum_{i=1}^{m}\frac{|V(G_i)|-1}{2}=2\quad\mbox{  and  }\quad |B|+\sum_{i=1}^m|V(G_i)|=n-1\ge 10\mbox{.}$$
Without loss of generality, let $|V(G_1)|\ge|V(G_2)|\ge\dots\ge|V(G_m)|$.
Thus $|B|\le 2$. Also, $H_u\subset K[B]\cup K[B,V(H_u)-B]\cup \sum_{i=1}^{m}K[V(G_i)]$. 
\begin{claim}\label{S_3-1}
$1\le|B|\le2$.
\end{claim}
\begin{proof}
If $|B|=0$, then $E(H_u)\subset \sum_{i=1}^{m}E(K[V(G_i)])$ and $\sum_{i=1}^{m}\frac{|V(G_i)|-1}{2}=2$. Note that $\frac{|V(G_i)|-1}{2}$ is a non-negative integer for any $i\in [m]$. so it is easy to see that either $|V(G_1)|=5$ and $|V(G_j)|=1$ for $j>1$ or $|V(G_1)|=|V(G_2)|=3$ and $|V(G_j)|=1$ for $j>2$. This implies $d_{H}(u)=|E(H_u)|=\sum_{i=1}^{m}|E(K[V(G_i)])|\le 10<n\le \delta_1(H)$, a contradiction. This gives $1\le |B|\le 2$.
\end{proof}
\begin{claim}\label{S_3-2}
For any edge $xy\in E(H_u)$, there is no copy of $2K_2$ in $H_x-\{u,y\}$. Moreover, $|E(H_x-\{u,y\})|\le n-4$.
\end{claim}
\begin{proof}
If there exists a set of two disjoint edges $\{a_1a_2, b_1b_2\}\subset E(H_x-\{u,y\})$ as a $2K_2$ in $H_x-\{u,y\}$, then the three edges $xyu, xa_1a_2, xb_1b_2\in E(H)$ form a copy of $S_3$, a contradiction.
Hence, the only non-empty component of $H_x-\{u,y\}$ must be a $K_3$ or a $K_{1,t}$ for some $1\le t\le n-4$. This gives $|E(H_x-\{u,y\})|\le n-4$. .
\end{proof}
\begin{claim}\label{S_3-3}
Let $v\in V(H_u)$ and $d_{H_u}(v)\ge 5$. Pick any two vertices $x,y\in N_{H_u}(v)$. If $d_{H_x-\{u\}}(v)\ge 1$, then $d_{H_y-\{u\}}(v)\le 1$. Moreover, $\max\{d_{H_u}(x),d_{H_u}(y)\} \ge 3$.

\end{claim}
\begin{proof}
Otherwise, suppose $d_{H_x-\{u\}}(v)\ge 1$ and $d_{H_y-\{u\}}(v)\ge 2$. then we can pick an edge $va_1\in H_x-\{u\}$ and another edge $va_2\in H_y-\{u\}$ with $a_2\neq a_1$. Since $d_{H_u}(v)\ge 5$, we can also pick a vertex $a_3\in N_{H_u}(v)$ with $a_3\neq a_1,a_2,x,y$. Then the three edges $va_1x, va_2y, va_3u\in E(H)$ form a copy of $S_3$, a contradiction. 

To prove $\max\{d_{H_u}(x),d_{H_u}(y)\}\ge 3$, note that $d_{H}(z)=d_{H_u}(z)+d_{H_z-\{u\}}(v)+|E(H_z-\{u,v\})|$ for $z\in\{x,y\}$. By Claim~\ref{S_3-2}, $|E(H_z-\{u,v\})|\le n-4$ for $z=x,y$.
Hence, for $z\in\{x,y\}$,
$$n\le \delta_1(H)\le d_{H}(z)\le n-4+d_{H_u}(z)+d_{H_z-\{u\}}(v).$$
Now if $d_{H_u}(z)\le 2$ for $z=x,y$, then $n\le n-2+d_{H_z-\{u\}}(v)$, which means $d_{H_z-\{u\}}(v)\ge 2$ for $z=x,y$. This is impossible by the proof above.

\end{proof}

Now by Claim~\ref{S_3-1}, $1\le |B|\le 2$. 

If $|B|=1$, let $B=\{v\}$. By $E(H_u)\subset K[B,V(H_u)-B]\cup \sum_{i=1}^{m}K[V(G_i)]$ and $\sum_{i=1}^{m}\frac{|V(G_i)|-1}{2}=1$, we have $|V(G_1)|=3$, $|V(G_j)|=1$ for $j>1$ and $E(H_u)=$ $E(H_u[B,V(H_u)-B])\cup E(G_1)$. 
Since $E(H_u)=d_{H}(u)\ge\delta_1(H)\ge n$, $d_{H_u}(v)=|E(H_u[B,V(H_u)-B])|\ge n-|K{V(G_1)}|=n-3>5=2+3$. Thus, we can pick two vertices $x,y\in N_{H_u}(v)\backslash V(G_1)$. Then $d_{H_u}(x)=d_{H_u}(y)=1$, contradicts to $\max\{d_{H_u}(x), d_{H_u}(y)\}\ge 3$ by Claim~\ref{S_3-3}.

If $|B|=2$, let $B=\{v_1,v_2\}$.  Similarly, we get $|V(G_j)|=1$ for any $j\in[m]$ and $E(H_u)=E(G[B])\cup E(G[B,V(H_u)-B])$. This means $d_{H_u}(z)\le 2$ for any $x\in V(H_u)\backslash\{v_1,v_2\}$ and $11\le n\le\delta_1(H)\le |E(H_u)|\le d_{H_u}(v_1)+d_{H_u}(v_2)$. By Pigeonhole Principle, one of $v_1$ and $v_2$, say $v_1$, has $d_{H_u}(v_1)\ge\frac{11}{2}>5$. So we can pick two vertices $x,y\in N_{H_u}(v_1)\backslash\{v_2\}$ and get a contradiction similarly by Claim~\ref{S_3-3}.
\end{proof}

\subsection{$GS_3$}
\begin{proof}[Proof of the lower bound of (6)]
Consider the graph $F$ with vertex set $\{0\} \cup [n-1]$. Let $B_i=\{2i-1, 2i\}\cap[n-1]$, for $i\in[\lceil\frac{n-1}{2} \rceil]$ and $\mathcal{B}=\{B_i: i\in [\lfloor\frac{n-1}{2} \rfloor]\}$.
The edge set of $F$ is
    $$E(F)=(\{\{0\}\} \vee\mathcal{B})\cup\bigcup_{\{i,j,k\}\in\binom{[\lceil\frac{n-1}{2} \rceil]}{3}} \left( \binom{B_i}{1}\vee\binom{B_j}{1}\vee\binom{B_k}{1}\right).$$ 
Clearly, $\delta_1(F)=\lfloor (n-1)/2 \rfloor$, and there is no copy of $GS_3$ covering $0$.
\end{proof}

\begin{proof}[Proof of the upper bound of (6)]
    Suppose that $H$ is a $3$-graph on $n\ge 13$ vertices with $\delta_1(H)\ge\lfloor (n-1)/2 \rfloor+1\ge 7$ and $u$ is a vertex in $H$ not covered by $GS_3$. 
    Apparently, $H_u$ contains at least one vertex $x$ such that $d_{H_u}(x) \ge \frac{2(\lfloor (n-1)/2 \rfloor+1\rceil)}{n-1}=2$.

\begin{claim}\label{GS_3-1}
    $H_u$ contains no copy of $K_{1,2} \cup K_2$.
\end{claim}
\begin{proof}
    Assume that $\{x_1x_2,x_2x_3,y_1y_2\}$ is a copy of $K_{1,2} \cup K_2$ in $H_u$, then $ux_1x_2,ux_2x_3,uy_1y_2$ form a $GS_3$ covering $u$.
\end{proof}

\begin{claim}
    The only non-empty component of $H_u$ is a star.
\end{claim}
\begin{proof}
   Suppose not and let $x$ be the vertex with maximum degree in $H_u$. Let $N_{H_u}[x]=N_{H_u}(x)\cup\{x\}$. Apparently, $|N_{H_u}[x]| \ge 3$ and any edge in $H_u$ shares at least one vertex in $N_{H_u}[x]$. Otherwise, there would be a copy of $K_{1,2} \cup K_2$ in $H_u$, which is a contradiction by Claim~\ref{GS_3-1}. So we can assume that all edges are incident with $N_{H_u}[x]$.
Suppose $N_{H_u}[x]=\{x,y_1,y_2,\dots, y_d\}$ where $d=d_{H_u}(x)\ge 2$.
    
    If $|N_{H_u}[x]|\ge 4$, pick an edge $wv$ with $x\notin wv$ (since $H_u$ is not a star), then $wv,xy_i,xy_j$ form a copy of $K_{1,2} \cup K_2$, where we pick $y_i,y_j\in N_{H_u}[x]\backslash\{w,v\}$. This is a contradiction.
If $|N_{H_u}[x]| = 3$, we have $\max\{d_{H_u}(y_1),d_{H_u}(y_2)\}\ge \lceil1+\frac{|E(H_u)|-2}{2}\rceil\ge 4$. Without loss of generality, suppose $d_{H_u}(y_1)\ge 4$. We can pick two vertices $z_1$and $z_2$ with $z_1,z_2\in N_{H_u}(y_1)\backslash\{x,y_2\}$. Then $y_1z_1,y_1z_2,xy_2$ form a copy of $K_{1,2} \cup K_2$, a contradiction. 
\end{proof}
    Now we can asuume that the only non-empty component of $H_u$ is $K[\{v\}, V_0]$ for some $v\in V(H_u)$ and $V_0\subset V(H_u)\backslash\{v\}$. Note that $|V_0|=d_{H}(u)\ge \lfloor (n-1)/2 \rfloor+1\ge 7$. 
If there exists an edge $e\in E(H-\{u\})$ with $v\in e$, we can pick $2$ vertices $v_1.v_2\in V_0\backslash e$. Hence we get a contradiction since $e$, $uvv_1$, $uvv_2$ form a copy of $GS_3$ covering $u$. 
    
    If there is no edge $e\in E(H-\{u\})$ with $v\in e$, then $d_{H}(\{u,v\})=d_{H}(u)=d_{H}(v)>0$ and $\delta_1(H-\{u,v\}) \ge \delta_1(H)-1 \ge  \lfloor (n-1)/2 \rfloor= \lfloor ((n-3)/2 \rfloor+1$. 
We now pick $w\in N_{H}(\{u,v\})$. Note that $|E((H-\{u,v\})_w)|\ge \delta_1(H-\{u,v\})\ge  \lfloor ((n-3)/2 \rfloor+1$. So we can find a vertex $x$ such that $d_{(H-\{u,v\})_w}(x) \ge \frac{2(\lfloor (n-3)/2 \rfloor+1\rceil)}{n-3}=2$. Pick $x_1, x_2\in N_{(H-\{u,v\})_w}(x)$, we get a copy of $GS_3$ in $H$ with edge set $\{uvw, wxx_1, wxx_2\}$ covering $u$. 
\end{proof}

\section{Concluding remarks}
In this paper, we determine the exact values of $c_1(F_5)$ and $c_1(n, F)$ for $F=LP_3$, $TP_3$, $K_{1,1,3}$, $S_3$, $GP_3$, $GS_3$. These results, together with some known ones, complete the $1$-degree thresholds for all possible coverings by a connected $3$-graph with $3$ edges.

For $3$-graphs $F$ with more than $3$ edges, however, we almostly have no non-trivial exact results for $c_1(F)$.

For the $2$-degree thresholds, one can easily check that:  $c_2(n, F)$ is a small constant for any mentioned connected $3$-graph $F$ with $3$ edges (except for $K_{4}^{(3)-}$ done by \cite{co-Our}). For example, 
\begin{itemize}
\item (Tang, Ma and Hou~\cite{Our}) For $n\ge 6$, $c_2(n, C_{6}^{(3)})=1$; 
\item (Gu, Wang~\cite{GW}) For $n\ge 5$, $c_2(n, F_5)\in\{1,2\}$ and $c_2(n, F_5)=2$ if and only if $n\equiv 1 \mod 3$ and $n\ge 10$; for $n\ge 8$, $c_2(n, LP_3)=1$; for $n\ge 7$, $c_2(n, S_3)\le 1$.
\end{itemize}
Hence, it seems to be more intresting to consider $c_1(n, F)$ and $c_1(F)$ than $c_{2}(n,F)$ and $c_2(F)$ for small $3$-graphs $F$.

\end{document}